\documentclass{article}
\usepackage[T1]{fontenc}
\usepackage{amsthm, fullpage}
\usepackage{amsmath}
\usepackage{amssymb}
\usepackage{amsfonts}
\usepackage{eucal}
\usepackage[all]{xy}
\usepackage[frenchb, english]{babel}
\usepackage{pslatex}
\usepackage{hyperref}

\newtheorem{prop}{Proposition}[section]
\newtheorem{theo}[prop]{Théor\`eme}
\newtheorem*{theo**}{Théorème}
\newtheorem{coro}[prop]{Corollaire}
\newtheorem*{conj*}{Conjecture}
\newtheorem{lemm}[prop]{Lemme}
\newtheorem{lemm*}{Lemme}[prop]

\theoremstyle{definition}

\newtheorem{vide}[prop]{}
\newtheorem{defi}[prop]{Définition}
\newtheorem*{defi*}{Définition}

\newtheorem{exem}[prop]{Exemples}

\numberwithin{equation}{prop}

\newcommand{\riso}{ \overset{\sim}{\longrightarrow}\, }
\newcommand{\liso}{ \overset{\sim}{\longleftarrow}\, }

\renewcommand{\sp}{\mathrm{sp}}

\newcommand{\RR}{{\mathcal{R}}}
\newcommand{\FF}{{\mathcal{F}}}

\newcommand{\E}{{\mathcal{E}}}
\newcommand{\G}{{\mathcal{G}}}

\newcommand{\D}{{\mathcal{D}}}

\renewcommand{\O}{{\mathcal{O}}}

\newcommand{\V}{\mathcal{V}}

\newcommand{\Y}{\mathcal{Y}}
\newcommand{\ZZ}{\mathcal{Z}}
\newcommand{\X}{\mathfrak{X}}

\newcommand{\U}{\mathfrak{U}}

\newcommand{\R}{\mathbb{R}}
\newcommand{\Q}{\mathbb{Q}}
\newcommand{\Z}{\mathbb{Z}}

\newcommand{\hdag}{  \phantom{}{^{\dag} }    }

\begin{document}

\selectlanguage{frenchb}

\title{Stabilité de l'holonomie sans structure de Frobenius: cas des courbes}
\author{Daniel Caro}

\date{}
\maketitle

\selectlanguage{english}
\begin{abstract}
By using Christol and Mebkhout's algebrization and finiteness theorem, 
we prove that in the case of smooth curves, 
Berthelot's strongest conjecture on the stability of holonomicity 
is still valid  without Frobenius structure 
but under some non-Liouville type hypotheses.

\end{abstract}

\selectlanguage{frenchb}

\tableofcontents

\section*{Introduction}
Soit $\V$ un anneau de valuation discrète complet d'inégales caractéristiques $(0,p)$,
de corps résiduel $k$ supposé parfait, de corps des fractions $K$.
Soient $\X$ un $\V$-schéma formel lisse, $T_0$ un diviseur de la fibre spéciale $X_0$ de $\X$, $\U$ l'ouvert de $\X$ complémentaire de $T_0$.
Berthelot a construit le faisceau sur $\X$ des opérateurs différentiels de niveau fini et d'ordre infini
noté $\D ^\dag _{\X,\Q}$ ; ce dernier correspondant à la tensorisation par $\Q$ (indiqué par l'indice $\Q$) du complété faible $p$-adique (indiqué par le symbole {\og $\dag$\fg})
du faisceau classique $\D  _{\X}$ des opérateurs différentiels sur $\X$.
Enfin, en ajoutant des singularités surconvergentes le long de $T_0$, il construit le faisceau 
$\D ^\dag _{\X} (\hdag T_0)_{\Q}$ sur $\X$ (voir \cite{Be1} ou \cite{Beintro2}).
On désigne par $F\text{-} D ^{\mathrm{b}} _{\mathrm{coh}}(\D ^\dag _{\X} (\hdag T_0)_{\Q})$
la catégorie des complexes de $\D ^\dag _{\X} (\hdag T_0)_{\Q}$-modules à cohomologie cohérente et bornée
munie d'une structure de Frobenius, i.e. d'un isomorphisme $\D ^\dag _{\X} (\hdag T_0)_{\Q}$-linéaire
de la forme $F ^* (\E) \riso \E$ avec $F^*$ désignant l'image inverse par l'endomorphisme (ou une puissance) du Frobenius absolu 
$X _0\to X_0$. 
 
Soit $\E\in F\text{-} D ^{\mathrm{b}} _{\mathrm{coh}}(\D ^\dag _{\X} (\hdag T_0)_{\Q})$.
Berthelot a conjecturé (voir \cite[5.3.6.D)]{Beintro2}) que
le $F$-complexes $\E$ est 
holonome si et seulement
si $\E |\U$ est 
holonome.
Cette conjecture a été validée dans le cas où $X_0$ est une courbe propre (voir \cite{caro_courbe-nouveau})
puis une courbe (voir \cite{caro-Tsuzuki-2008}) et enfin une variété projective(voir \cite{caro-stab-holo}).

Nous nous intéressons ici à une extension sans structure de Frobenius de cette conjecture. 
D'une part, via la caractérisation homologique de l'holonomie de Virrion (voir \cite{virrion}), il est possible d'étendre de manière naturelle la notion 
d'holonomie en nous affranchissant de la structure de Frobenius. 
Cependant, si on ne fait aucune hypothèse supplémentaire qui permettent d'exclure les problèmes liés aux nombres de Liouville, 
la conjecture de Berthelot sans structure de Frobenius est fausse. 
En effet, lorsque $X_0$ est propre, le problème est qu'il existe des $\D ^\dag _{\X} (\hdag T_0)_{\Q}$-modules cohérents, 
$\O _{\X} (\hdag T_0)_{\Q}$-cohérents dont les espaces de cohomologie $p$-adique ne sont pas de dimension finie 
et qui ne sont par conséquent même pas $\D ^\dag _{\X,\Q}$-cohérents
(par exemple, il suffit de considérer l'exemple donné par Berthelot à la fin de \cite{Be1}).
Mais, 
nous disposons de théorèmes d'algébrisation et de finitude de Christol-Mebkhout pour les isocristaux surconvergents sur les courbes lisses 
qui satisfont à certaines hypothèses de non Liouvillité  (voir l'introduction de \cite{christol-MebkhoutIV}).
Ces hypothèses de non Liouvillité 
sont résumés ici via l'appellation (d'origine contrôlée) {\og propriété $(NL\text{-}NL)$\fg}.
Par exemple, cette propriété $(NL\text{-}NL)$ est satisfaite lorsque l'on dispose d'une structure de Frobenius.

Dans ce papier, nous établissons que 
la conjecture de Berthelot reste valable sans structure de Frobenius dans le cas des courbes pour les complexes se dévissant en 
isocristaux surconvergents satisfaisant à la propriété $(NL\text{-}NL)$ (en fait, on prouve un peu mieux via le théorème \ref{conjD-sansFrob}). 

Ce papier se compose d'une première partie où nous faisons le lien entre les modules solubles utilisés par Christol-Mebkhout et 
les isocristaux surconvergents décris via les $\D$-modules arithmétiques de Berthelot. Dans une seconde partie, nous 
étudions les propriétés (en particulier l'holonomie) 
de l'image du foncteur introduit par Christol-Mebkhout (que l'on appelle ici {\og foncteur de daguification\fg})
dans leur théorème d'algébrisation. 
Nous donnons dans une troisième partie une interprétation via les $\D$-modules de Berthelot 
du théorème de finitude cohomologique de Christol et Mebkhout.
Enfin, nous établissons dans une dernière partie la conjecture de Berthelot sous les conditions $(NL\text{-}NL)$ dans le cas des courbes.
En particulier, nous obtenons une nouvelle preuve (on remplace ici le théorème de la réduction semi-stable de Kedlaya par ceux cités ci-dessus de Christol-Mebkhout) de la conjecture de Berthelot avec structure de Frobenius dans le cas des courbes.

\section*{Notations}

Soit $\V$ un anneau de valuation discrète complet d'inégales caractéristiques $(0,p)$,
de corps résiduel $k$ supposé parfait, de corps des fractions $K$.
Soient $X $ un $\V$-schéma projectif et lisse, 
$Z$ un diviseur ample de $X$ et $U$ l'ouvert de $X$ complémentaire de $Z$. 
On suppose $U=\mathrm{Spec}\, A$ affine et on note
$j\,:\, U \subset X$ l'inclusion canonique. 
On désigne par $X ^{\dag}$ et $U ^{\dag}$ les
$\V$-schémas formels faibles lisses déduits par complétion faible $p$-adique
de respectivement $X$ et $U$ (voir \cite{meredith-weakformalschemes}) ; 
par respectivement $\X$ et $\U$ les complétés $p$-adiques ; par $X _0$, $U _0$ et $Z _0$ 
les fibres spéciales respectives de $X$, $U$ et $Z$.
On suppose pour simplifier que $X _{0}$ est de dimension pure notée
$d _{X _{0}}$.
On dispose du morphisme structural 
$f\,:\, \X \to \mathrm{Spf} \V$,
des morphismes d'espaces annelés
$\epsilon\,:\, X ^{\dag} \to X$ 
et
$\epsilon _{U}\,:\, U ^{\dag} \to U$. 
On note $j _{0}\,:\, U ^{\dag} \subset X ^{\dag}$ l'inclusion canonique. 
Sans nuire à la généralité, nous supposerons que les $k$-schémas sont toujours réduits.

Nous utiliserons les notations usuelles sur les $\D$-modules arithmétiques que nous ne rappelons pas ici (e.g., voir \cite{Beintro2} et le premier chapitre de 
\cite{caro_surholonome}).
Le module des sections globales de faisceaux de la forme $\D _{*} ^{?}$ avec $*$ et $?$ quelconques
est noté $D _{*} ^{?}$, e.g.
$D _{U,\Q}:=\Gamma (U, \D _{U,\Q})$,
$D ^\dag _{U^\dag ,\Q}:=\Gamma (U^\dag , \D ^\dag _{U^\dag ,\Q})$,
etc. De plus, pour tout faisceau $\E$ de groupes, on pose $\E _{\Q}:= \E \otimes _{\Z} \Q$.

\section{Isocristaux surconvergents sur $U _0$}

\begin{vide}
[Foncteur $\sp _{+}$]
Soit $\E^\dag$ un $\D _{U^\dag,\Q}$-module cohérent, $\O _{U^\dag,\Q}$-cohérent.
On obtient un
$\D ^{\dag} _{\X} (\hdag Z _0)_{\Q}$-module cohérent 
en posant 
$$\sp _{+} (\E ^\dag):= \D ^{\dag} _{\X} (\hdag Z _0)_{\Q} \otimes _{j _{0*} \D _{U^\dag,\Q}} j _{0*} (\E ^{\dag}).$$
Ce foncteur a été défini dans \cite{caro_devissge_surcoh} dans un contexte un peu plus général.
Comme $Z$ est un diviseur ample de $X$, 
nous pouvons dans notre contexte utiliser les théorèmes de type $A$ de Noot-Huyghe 
(voir \cite{huyghe-comparaison}) pour les 
$\D ^{\dag} _{\X} (\hdag Z _0)_{\Q} $-modules cohérents.
Nous redémontrons et surtout affinons (via l'isomorphisme \ref{sp+-isom-isoc} qui nous sera utile) 
dans notre contexte (plus facile) les résultats analogues avec structure de Frobenius de 
\cite{caro-2006-surcoh-surcv} ou 
sans structure de Frobenius de \cite{caro-pleine-fidelite}.
\end{vide}

\begin{vide}
[Catégories d'isocristaux surconvergents sur $U _0$]
\begin{itemize}
\item On note $\mathrm{Isoc} ^\dag (A ^\dag _{K})$ la catégorie
des $A ^\dag _{K}$-modules cohérents munis d'une connexion surconvergente. 
D'après \cite[5]{Berig}, cette catégorie est canoniquement isomorphe à celle
des isocristaux surconvergents sur $U _{0}$. 

\item On note $\mathrm{Isoc} ^\dag (\O _{U ^{\dag},\Q})$ la catégorie
des $\D _{U^\dag,\Q}$-modules cohérents, $\O _{U^\dag,\Q}$-cohérents tels que 
le module de leurs sections globales soit un élément de $\mathrm{Isoc} ^\dag (A ^\dag _{K})$.

\item On note $\mathrm{Isoc} ^\dag (\X, Z _0 /K)$ la catégorie
des $\D ^{\dag} _{\X} (\hdag Z _0)_{\Q} $-modules cohérents, 
$\O _{\X} (\hdag Z _0)_{\Q} $-cohérents.
Cette catégorie est canoniquement (via les foncteurs quasi-inverses image directe et image inverse par le morphisme de spécialisation de $\X$) isomorphe à celle
des isocristaux surconvergents sur $U _{0}$.
\end{itemize}

\end{vide}

\begin{vide}
\label{eq-3Isoc}
Via les théorème de type $A$:

\begin{itemize}
\item Les foncteurs 
sections globales $\Gamma (U ^\dag, -)$
et 
$\D _{U^\dag,\Q}  \otimes _{D _{U,\Q} } -$
(ou $\O _{U^\dag,\Q} \otimes _{A_{\Q} } -$) 
 induisent des équivalences quasi-inverses
entre $\mathrm{Isoc} ^\dag (A ^\dag _{K})$
et $\mathrm{Isoc} ^\dag (\O _{U ^{\dag},\Q})$.

\item Les foncteurs 
sections globales $\Gamma (\X, -)$ 
et 
$\D ^{\dag} _{\X} (\hdag Z _0)_{\Q} \otimes _{D ^{\dag} _{\X} (\hdag Z _0)_{\Q}} -$
(ou $\D  _{\X} (\hdag Z _0)_{\Q} \otimes _{D  _{\X} (\hdag Z _0)_{\Q}} -$
ou $\O _{\X} (\hdag Z _0)_{\Q} \otimes _{A ^\dag _K} -$) 
induisent des équivalences quasi-inverses
entre $\mathrm{Isoc} ^\dag (A ^\dag _{K})$
et $\mathrm{Isoc} ^\dag (\X, Z _0 /K)$.

\end{itemize}

\end{vide}

\begin{prop}
\label{sp+-isom-isoc-prop}
Le foncteur $\sp _{+}$ se factorise en une équivalence de catégories
de la forme 
$\sp _{+}\,:\,\mathrm{Isoc} ^\dag (\O _{U ^{\dag},\Q})
\cong 
\mathrm{Isoc} ^\dag (\X, Z _0 /K)$.
Pour tout $\E ^\dag \in \mathrm{Isoc} ^\dag (\O _{U ^{\dag},\Q})$, 
on dispose en outre de l'isomorphisme canonique : 
\begin{equation}
\label{sp+-isom-isoc}
\Gamma ( \X, \sp _{+} (\E ^\dag) ) 
\riso 
\Gamma ( U ^\dag, \E ^\dag ) .
\end{equation}
\end{prop}

\begin{proof}
Soit $\E^\dag\in \mathrm{Isoc} ^\dag (\O _{U ^{\dag},\Q})$.
D'après le premier point de \ref{eq-3Isoc}, 
$E ^\dag := \Gamma (U ^{\dag}, \E ^\dag) \in \mathrm{Isoc} ^\dag (A ^\dag _{K})$ et
le morphisme canonique 
$\D _{U^\dag,\Q} \otimes _{D _{U^\dag,\Q} } E ^\dag
\to
\E^\dag$ est un isomorphisme. 
En choisissant une présentation finie de $E ^\dag$, 
on déduit alors de \cite[2.2.9]{caro_devissge_surcoh} l'isomorphisme canonique
$j _{0*} \D _{U^\dag,\Q} \otimes _{D _{U^\dag,\Q} } E ^\dag 
\riso j _{0*} (\E ^{\dag})$.
Il en résulte l'isomorphisme
\begin{equation}
\label{sp+isogen}
\sp _{+} (\E ^\dag) 
\riso 
\D ^{\dag} _{\X} (\hdag Z _0)_{\Q} \otimes _{D _{U^\dag ,\Q} } E ^\dag
=
\D ^{\dag} _{\X} (\hdag Z _0)_{\Q} \otimes _{D _{\X} (\hdag Z _0)_{\Q} } E ^\dag,
\end{equation}
l'égalité découlant de la formule 
$D _{\X} (\hdag Z _0)_{\Q} = D  _{U^\dag ,\Q}$ 
vérifiée par Noot-Huyghe (voir \cite{huyghe-comparaison}).
D'après \ref{eq-3Isoc}, $\G := \D  _{\X} (\hdag Z _0)_{\Q}
\otimes _{D _{\X} (\hdag Z _0)_{\Q} }
E ^\dag \in \mathrm{Isoc} ^\dag (\X, Z _0 /K)$.
Il en résulte que le morphisme canonique 
$\D ^\dag  _{\X} (\hdag Z _0)_{\Q}
\otimes _{\D _{\X} (\hdag Z _0)_{\Q} }
\G \to \G$ est un isomorphisme (e.g. il suffit de le voir en dehors de $Z$ 
et utiliser la version analogue sans diviseur de \cite[2]{Be0}).
On en déduit l'isomorphisme:
\begin{equation}
\D ^{\dag} _{\X} (\hdag Z _0)_{\Q} \otimes _{D _{\X} (\hdag Z _0)_{\Q} } E ^\dag 
\riso 
\D  _{\X} (\hdag Z _0)_{\Q}
\otimes _{D _{\X} (\hdag Z _0)_{\Q} }
E ^\dag.
\end{equation}
D'après le deuxième point de \ref{eq-3Isoc}, 
on bénéficie de l'isomorphisme dans $\mathrm{Isoc} ^\dag (\X, Z _0 /K)$:
$\D  _{\X} (\hdag Z _0)_{\Q}
\otimes _{D _{\X} (\hdag Z _0)_{\Q} }
E ^\dag
\riso 
\D ^\dag  _{\X} (\hdag Z _0)_{\Q}
\otimes _{D ^\dag _{\X} (\hdag Z _0)_{\Q} }
E ^\dag$.
En composant ces trois isomorphismes, on obtient l'isomorphisme qui nous permet de conclure: 
\begin{equation}
\label{sp+isogenFIN}
\sp _{+} (\E ^\dag) 
\riso 
\D ^\dag  _{\X} (\hdag Z _0)_{\Q}
\otimes _{D ^\dag _{\X} (\hdag Z _0)_{\Q} }
E ^\dag.
\end{equation}
\end{proof}

\section{Foncteur daguification}
\begin{vide}

On note $\mathrm{MC} (\O _{U^{\dag}, \Q} )$ la catégorie
 des $\D _{U^{\dag}, \Q}$-modules cohérents, $\O _{U^{\dag}, \Q}$-cohérents ; 
 $\mathrm{MC} (\O _{U, \Q} )$ la catégorie
 des $\D _{U, \Q}$-modules cohérents, $\O _{U, \Q}$-cohérents.
 On dispose du foncteur canonique : 
 \begin{equation}
 \label{foncteur-dag}
 \dag\,:\, 
 \mathrm{MC} (\O _{U, \Q} )
 \to 
 \mathrm{MC} (\O _{U^{\dag}, \Q} )
\end{equation}
défini par 
$\E \mapsto \E ^\dag :=  \O _{U^{\dag}, \Q} \otimes _{\epsilon ^{-1} \O _{U, \Q}} \epsilon ^{-1} \E$.
Ce foncteur est bien défini d'après le paragraphe \ref{foncteurdag-def} (plus précisément \ref{ThA-Edag}).
Dans la suite de ce chapitre, on se donne $\E \in \mathrm{MC} (\O _{U^{\dag}, \Q} )$.
On note $\E ^\dag$ son image par le foncteur \ref{foncteur-dag} et
on pose $E:= \Gamma (U, \E )$ et $E ^\dag := \Gamma (U^\dag, \E ^\dag)$.

\end{vide}

\begin{vide}
[Description du foncteur $\dag$]
\label{foncteurdag-def}
D'après les théorèmes de type $A$ sur 
les $\O _{U,\Q}$-modules ou $\D _{U,\Q}$-modules cohérents, 
on dispose des isomorphismes canoniques: 
\begin{equation}
\O _{U,\Q} \otimes _{A _\Q } E 
\riso
\D _{U,\Q} \otimes _{D _{U,\Q} } E 
\riso
\E.
\end{equation}
Comme le morphisme canonique
$\O _{U^{\dag}, \Q} \otimes _{\epsilon ^{-1} \O _{U, \Q}} 
\epsilon ^{-1} \D _{U,\Q} 
\to 
\D _{U^\dag,\Q}$
est un isomorphisme, 
comme le foncteur \ref{foncteur-dag} est exacte à droite, 
comme $E$ un $D _{U,\Q}$-module de présentation finie, 
on en déduit les isomorphismes canoniques:
\begin{equation}
\label{ThA-Edag}
\O _{U^\dag,\Q} 
\otimes _{A_{\Q} } 
E
\riso
\D _{U^\dag,\Q} 
\otimes _{D _{U,\Q} } 
E 
\riso \E ^\dag.
\end{equation}
Via le théorème de type $A$ sur 
les $\D _{U^\dag ,\Q}$-modules cohérents,
il en dérive
$D _{U^\dag ,\Q}
\otimes _{D _{U,\Q} } 
E 
\riso
E ^\dag$.

\end{vide}

\begin{lemm}
\label{j*Dcoh=coh}
Le $\D _{X,\Q}$-module $j _{*} (\E)$ est cohérent.
De plus, pour tout $l \not = d _{X _{0}}$, $\mathcal{E}xt ^{l} _{\D _{X,\Q}} (j _{*} (\E), \D _{X,\Q})=0$.
\end{lemm}

\begin{proof}
L'assertion est locale en $X$. On peut supposer $X$ affine, muni de coordonnées locales et $Z$ défini par une équation locale. 
D'après le théorème de type $A$ sur les $\D _{U,\Q}$-modules cohérents (resp. les $\O _{U,\Q}$-modules cohérents), 
$E:= \Gamma (U, \E)$ est un $\Gamma (U,  \D _{U,\Q})$-module cohérent (resp. $\Gamma (U, \O _{U,\Q})$-module cohérent).
Or, $\Gamma (U,  \D _{U,\Q})=\Gamma (U _K,  \D _{U_K})$ et $\Gamma (U,  \O _{U,\Q})=\Gamma (U _K,  \O _{U_K})$.
On note 
$\FF:= \D _{U_K}  \otimes _{\Gamma (U _K,  \D _{U_K})} E \liso \O _{U_K}  \otimes _{\Gamma (U _K,  \O _{U_K})} E$
le $\D _{U_K}$-module cohérent, $\O _{U_K}$-cohérent correspondant.

Comme $U_K$ est une variété sur $K$ qui est de caractéristique nulle, 
comme $\FF$ est un $\D _{U_K}$-module holonome (car $\O _{U_K}$-cohérent), 
alors par préservation de l'holonomie par image directe, 
$j _{K*} (\FF)$ est une $\D _{X_K}$-module holonome.
Comme l'entier $d _{X _{0}}$ est aussi la dimension de $X _K$, 
d'après les théorèmes de type $A$ et $B$, il en résulte que 
$\Gamma (X _K, j _{K*} (\FF))$ 
est un 
$D _{X_K}$-module cohérent
et que, pour tout $l \not = d _{X _{0}}$, $\mathrm{Ext} ^{l} _{D _{X_K}} (\Gamma (X _K, j _{K*} (\FF)), D _{X_K})=0$.
Comme 
$D _{X,\Q}=D _{X_K}$
et 
$\Gamma (X , j _{*} (\E)) = \Gamma (X _K, j _{K*} (\FF))$,
alors
$\Gamma (X , j _{*} (\E))$ est $D _{X,\Q}$-cohérent. 
La quasi-cohérence de $j _{*} (\E)$ nous permet alors d'en conclure que 
$j _{*} (\E)$ est $\D _{X,\Q}$-cohérent. 
Via les théorèmes de type $A$ et $B$, il en résulte que 
$\mathcal{E}xt ^{l} _{\D _{X,\Q}} (j _{*} (\E), \D _{X,\Q})=0$
si et seulement si 
$\mathrm{Ext} ^{l} _{D _{X,\Q}} (\Gamma (X , j _{*} (\E)), D _{X,\Q})=0$.
D'où le résultat.
\end{proof}

\begin{vide}
\label{defpropspssdag}
Grâce au lemme \ref{j*Dcoh=coh}, on obtient un $\D _{X^\dag,\Q} $-module cohérent en posant
\begin{equation}
\label{spssdag}
j ^\dag _* (\E):=
\D _{X^\dag,\Q} \otimes _{\epsilon ^{-1}\D _{X,\Q}} \epsilon ^{-1} (j _* \E).
\end{equation} 
\end{vide}

\begin{defi}
[Holonomie]
D'après \cite{caro-holo-sansFrob}, un $\D ^{\dag} _{\X,\Q} $-module $\G$ est par définition holonome
s'il est $\D ^{\dag} _{\X,\Q} $-cohérent et si,
pour tout entier $l \not = d _{X _{0}}$,  
$\mathcal{E}xt ^{l} _{\D ^{\dag} _{\X,\Q}} (
\G, \D ^{\dag} _{\X,\Q})=0$.
Lorsque l'on dispose d'une structure de Frobenius, on retrouve la définition de Berthelot. 
\end{defi}

\begin{lemm}
\label{j*dag-holo}
Avec les notations de \ref{defpropspssdag}, 
le $\D ^{\dag} _{\X,\Q} $-module 
$\D ^{\dag} _{\X,\Q} 
\otimes _{\D _{X^\dag,\Q}}j ^\dag _* (\E)$
est holonome.
\end{lemm}

\begin{proof}
 La $\D ^{\dag} _{\X,\Q} $-cohérence 
de $\G := \D ^{\dag} _{\X,\Q} 
\otimes _{\D _{X^\dag,\Q}}j ^\dag _* (\E)$ se déduit du lemme \ref{j*Dcoh=coh}. 
Soit $l \not = d _{X _{0}}$ un entier. 
Pour vérifier que 
$\mathcal{E}xt ^{l} _{\D ^{\dag} _{\X,\Q}} (
\D ^{\dag} _{\X,\Q} 
\otimes _{\D _{X^\dag,\Q}}j ^\dag _* (\E), \D ^{\dag} _{\X,\Q})=0$, 
on se ramène au cas où $X$ est affine.
Comme l'extension $D _{X,\Q} \to D ^{\dag} _{\X} (\hdag Z _0)_{\Q}$ est plate, 
cela résulte alors de \ref{j*Dcoh=coh}.
\end{proof}

\begin{lemm}
Avec les notations de \ref{defpropspssdag},
on dispose de l'isomorphisme canonique 
de $\D ^{\dag} _{\X} (\hdag Z _0)_{\Q}$-modules cohérents: 
\begin{equation}
\label{jdagE-Edag}
\D ^{\dag} _{\X} (\hdag Z _0)_{\Q} 
\otimes _{\D _{X^\dag,\Q}}j ^\dag _* (\E) 
\riso \sp _{+} (\E ^\dag).
\end{equation}
\end{lemm}

\begin{proof}
De manière analogue à \ref{ThA-Edag}, on vérifie que le morphisme canonique 
$\E ^\dag \to \D _{U^\dag,\Q} \otimes _{\epsilon ^{-1}\D _{U,\Q}} \epsilon ^{-1} (\E)$
est un isomorphisme. 
Il en résulte $j _{0} ^{*} (j ^\dag _* (\E)) \riso 
\D _{U^\dag,\Q} \otimes _{\epsilon ^{-1}\D _{U,\Q}} \epsilon ^{-1} (\E)
\liso \E ^\dag$.
Par adjonction, on obtient alors le premier morphisme
$j ^\dag _* (\E) \to j _{0*} (\E ^\dag)\to \sp _{+} (\E ^\dag)$.
On en déduit par extension le morphisme \ref{jdagE-Edag} voulu.
Enfin, ce morphisme est un isomorphisme car il l'est en dehors de $Z$.
\end{proof}

\section{Conditions $(DNL\text{-}NL)$ et $(NL\text{-}NL)$}

Afin d'utiliser les définitions et résultats de Christol et Mebkhout, nous supposons dans cette section que 
$X _{0}$ est de dimension pure égale à $1$. 
Pour simplifier, on suppose en outre que tous les points fermés de $Z _{k}$ soient $k$-rationnel.
Soit $\E^\dag\in \mathrm{Isoc} ^\dag (\O _{U ^{\dag},\Q})$.
Pour tout point fermé $x$ de $Z _{k}$, nous notons $\E ^{\dag}_x$ le $\RR _{K}$-module associé 
à $\E ^\dag$ en $x$ ($\E ^{\dag}_x$ est une extension du module des sections sur le tube $]x[ _\X$ de $x$ dans $\X$ de l'isocristal surconvergent $U _{0}$ associé à $\E ^\dag$, 
en effet le tube $]x[ _\X$ est isomorphe à la boule unité ouverte). 
Considérons les propriétés suivantes qui apparaissent dans \cite[5]{christol-MebkhoutIV} (nous y renvoyons le lecteur pour plus de détails, ainsi qu'à \cite{christol-MebkhoutII}):
\begin{enumerate}
\item \label{hypo-E-isoc-i} pour tout point fermé $x$ de $X _{k}$, $\E ^{\dag}_x$ 
a la propriété $(DNL)$ ;
\item \label{hypo-E-isoc-ibis} pour tout point fermé $x$ de $X _{k}$, $\E ^{\dag}_x$ 
a la propriété $(NL)$ ;
\item \label{hypo-E-isoc-ii} pour tout point fermé $x$ de $X _{k}$, 
$\mathrm{End} _{\RR _{K}} ((\E ^{\dag} _x) _{>0}) $ a la propriété $(NL)$.
\end{enumerate}

\begin{defi}
\label{defi-NL-NL-isoc}
\begin{itemize}
\item On dira que $\E ^{\dag}$ vérifie la propriété $(DNL\text{-}NL)$ (resp. $(NL\text{-}NL)$) si 
les conditions \ref{hypo-E-isoc-i} et \ref{hypo-E-isoc-ii} (resp. \ref{hypo-E-isoc-ibis} et \ref{hypo-E-isoc-ii}) sont satisfaites.
Rappelons que la condition \ref{hypo-E-isoc-ibis} est plus forte que la \ref{hypo-E-isoc-i}.
La propriété $(NL\text{-}NL)$ est donc plus forte que $(DNL\text{-}NL)$. 

\item On dira qu'un objet de $\mathrm{Isoc} ^\dag (\X, Z _0 /K)$ vérifie la propriété $(DNL\text{-}NL)$ (resp. $(NL\text{-}NL)$)
si l'objet (à isomorphisme près) de $\mathrm{Isoc} ^\dag (\O _{U ^{\dag},\Q})$ 
correspondant via l'équivalence $\sp _{+}$ de \ref{sp+-isom-isoc-prop} 
vérifie la propriété $(DNL\text{-}NL)$ (resp. $(NL\text{-}NL)$).

\end{itemize}

\end{defi}

\begin{vide}
[Théorèmes d'algébrisation et de finitude de Christol-Mebkhout]
\label{hypo-E-isoc}
On dispose des théorèmes de Christol et Mebkhout (voir \cite[5]{christol-MebkhoutIV}):
\begin{enumerate}
\item Si $\E ^{\dag}$ vérifie la propriété $(DNL\text{-}NL)$ alors, 
quitte à faire une extension finie du corps de base $K$, 
$\E ^\dag$ est  
dans l'image essentielle du foncteur $\dag$ de \ref{foncteur-dag}. 

\item 
\label{espace-dimfinie}
Si $\E ^{\dag}$ vérifie la propriété $(NL\text{-}NL)$ alors les $K$-espaces de cohomologie de de Rham $p$-adique 
\begin{equation}
\mathrm{Hom}_{\D _{U ^\dag\Q} } (\O _{U ^\dag ,\Q}, \E ^\dag), 
\mathrm{Ext} ^1 _{\D _{U ^\dag, \Q}} (\O _{U ^\dag ,\Q}, \E ^\dag)
\end{equation}
sont de dimension finie.

\end{enumerate}

\end{vide}

Traduisons à présent le théorème de finitude de Christol et Mebkhout dans le langage des
$\D$-modules arithmétiques de Berthelot: 

\begin{prop}
\label{Traduc-NLNL}
Si $\E ^{\dag}$ vérifie la propriété $(NL\text{-}NL)$,
les espaces de cohomologie de
$f _{+} (\sp _{+} (\E ^\dag))$
sont de dimension finie. 
\end{prop}

\begin{proof}
1) On a :
$\R \mathrm{Hom}_{\D _{U ^\dag\Q} } (\O _{U ^\dag ,\Q}, \E ^\dag)
\riso
\R \Gamma (U ^\dag, -) \circ  \R \mathcal{H} om_{\D _{U ^\dag\Q} } (\O _{U ^\dag ,\Q}, \E ^\dag)$.
D'après le théorème de type $B$, les termes du complexe
$ \R \mathcal{H} om_{\D _{U ^\dag\Q} } (\O _{U ^\dag ,\Q}, \E ^\dag)
\riso
\Omega ^\bullet _{U ^\dag ,\Q} \otimes _{\O _{U ^\dag ,\Q}}  \E ^\dag $
sont acycliques pour le foncteur
$\Gamma (U ^\dag, -)$.
Il en résulte:
$\R \mathrm{Hom}_{\D _{U ^\dag\Q} } (\O _{U ^\dag ,\Q}, \E ^\dag)
\riso \R \mathrm{Hom}_{D _{U ^\dag\Q} } (A ^\dag _{\Q}, E ^\dag)$.

\noindent 1') De la même façon, grâce aussi à \ref{sp+-isom-isoc}, on obtient le premier isomorphisme: 
$$\R \mathrm{Hom}_{D _{U ^\dag\Q} } (A ^\dag _{\Q}, E ^\dag)
\riso 
\R \mathrm{Hom}_{\D  _{\X} (\hdag Z _0)_{\Q}} (\O  _{\X} (\hdag Z _0)_{\Q}, \sp _{+}(\E ^\dag))
\riso 
\R \mathrm{Hom}_{\D  _{\X,\Q}} (\O  _{\X,\Q}, \sp _{+}(\E ^\dag)).$$
2) Comme l'extension $\D  _{\X,\Q} \to \D ^\dag  _{\X,\Q}$ est plate, 
comme le morphisme canonique
$\D ^\dag  _{\X,\Q}
\otimes _{\D  _{\X,\Q}} \O  _{\X,\Q}
\to
\O  _{\X,\Q}$ est un isomorphisme,
on obtient alors le premier isomorphisme:
$$\R \mathrm{Hom}_{\D  _{\X,\Q}} (\O  _{\X,\Q}, \sp _{+} (\E ^\dag))
\riso
\R \mathrm{Hom}_{\D ^\dag _{\X,\Q}} (\O  _{\X,\Q}, \sp _{+} (\E ^\dag) )
\riso
f _{+} (\sp _{+} (\E ^\dag)) [-d _X].$$
3) En composant les isomorphismes de 1), 1') et 2), on obtient: 
$$\R \mathrm{Hom}_{\D _{U ^\dag\Q} } (\O _{U ^\dag ,\Q}, \E ^\dag)
\riso 
f _{+} (\sp _{+} (\E ^\dag)) [-d _X].$$
On en déduit la proposition grâce à \ref{hypo-E-isoc}.\ref{espace-dimfinie}.
\end{proof}

\section{Condition $(NL\text{-}NL)$ et stabilité de l'holonomie}
Nous supposons dans ce chapitre que 
$X _{0}$ est de dimension pure égale à $1$.

\begin{prop}
\label{2.3.2-courbe-bis}
Soit $\G \in D ^{\mathrm{b}} _{\mathrm{coh}}(\D ^\dag _{\X,\Q})$ (resp. $\G \in D ^{\mathrm{b}} _{\mathrm{hol}}(\D ^\dag _{\X,\Q})$) 
tel que 
les espaces de cohomologie de 
$f _{+} (\G(\hdag Z _0))$ soient de dimension finie. 
Alors $\G (\hdag Z _0)\in D ^{\mathrm{b}} _{\mathrm{coh}}(\D ^\dag _{\X,\Q})$
(resp. $\G (\hdag Z _0)\in D ^{\mathrm{b}} _{\mathrm{hol}}(\D ^\dag _{\X,\Q})$).
\end{prop}

\begin{proof}
La preuve est identique à celle de \cite[2.3.2]{caro_courbe-nouveau} : 
si deux des complexes d'un triangle distingué sont cohérents (resp. holonomes) alors le troisième l'est aussi. 
En appliquant le foncteur $f _{+}$ au triangle de localisation de $\G$ relative à $Z _0$, 
il en résulte la cohérence (resp. holonomie) de 
$f _{+} \R \underline{\Gamma} ^\dag _{Z _0}  (\G )$. 
Or, via le théorème de Berthelot-Kashiwara, cette cohérence (resp. holonomie) est équivalente à 
celle de $\R \underline{\Gamma} ^\dag _{Z _0}  (\G )$.
Grâce au triangle de localisation de $\G$ relative à $Z _0$, 
il en résulte alors celle de $\G$.
\end{proof}

\begin{prop}
\label{desc-hol-chgtbase}
Soit $\V \to \V'$ une extension finie d'anneau de valuation discrètes complets d'inégales caractéristiques $(0,p)$.
On pose $\X ' \times _{\mathrm{Spf} (\V)} \mathrm{Spf} (\V')$ le $\V'$-schéma formel lisse déduit par changement de base
et $\alpha \,:\, \X' \to \X$ la projection canonique. 
Soit $\G$ un $\D ^\dag _{\X} (\hdag Z _0)_{\Q}$-module cohérent. Alors $\G$ est $\D ^\dag _{\X,\Q}$-cohérent (resp. holonome) si et seulement
si $\alpha ^{*} (\G)$ est $\D ^\dag _{\X',\Q}$-cohérent (resp. holonome). 
\end{prop}

\begin{proof}
Si $\G$ est $\D ^\dag _{\X,\Q}$-cohérent (resp. holonome), il est alors immédiat (voir \cite{Beintro2}) que 
$\alpha ^{*} (\G)$ soit $\D ^\dag _{\X',\Q}$-cohérent (resp. holonome). 
Réciproquement, supposons que 
$\alpha ^{*} (\G)$ soit $\D ^\dag _{\X',\Q}$-cohérent.
Comme l'assertion est locale, on se ramène au cas où $\X$ est affine. 
Posons $G:= \Gamma (\X,\G)$ et $Z ' _0:= \alpha ^{-1} (Z_0)$.

0) Comme dans le cas des $\D$-modules en caractéristique nulle, 
on vérifie que le morphisme canonique 
$\D ^\dag _{\X'} (\hdag Z '_0)_{\Q} \to \alpha ^{*} \D ^\dag _{\X} (\hdag Z _0)_{\Q}$
est en fait un isomorphisme et que le morphisme que l'on en déduit
$\alpha ^{-1} \D ^\dag _{\X} (\hdag Z _0)_{\Q}
\to
\D ^\dag _{\X'} (\hdag Z '_0)_{\Q} $ 
est un morphisme d'anneaux (voir \cite[2.2.2]{Beintro2}). 
On obtient ainsi les extensions 
$D ^\dag _{\X} (\hdag Z _0)_{\Q}
\to
D ^\dag _{\X'} (\hdag Z '_0)_{\Q} $ et 
$D ^\dag _{\X, \Q} \to D ^\dag _{\X', \Q} $.

I.1) 
Comme $\G$ est $\D ^\dag _{\X} (\hdag Z _0)_{\Q}$-cohérent,
par théorème de type $A$:
$\Gamma (\X',\alpha ^{*} (\G))
\riso 
D ^\dag _{\X'} (\hdag Z _0')_{\Q} \otimes _{D ^\dag _{\X} (\hdag Z _0)_{\Q}} G$.
Comme le morphisme canonique $D ^\dag _{\X',\Q} \otimes_{D ^\dag _{\X,\Q}} 
D ^\dag _{\X} (\hdag Z _0)_{\Q}
 \to D ^\dag _{\X'} (\hdag Z _0')_{\Q}$ est 
 un isomorphisme,
 il en résulte 
 $\Gamma (\X',\alpha ^{*} (\G))
\riso 
D ^\dag _{\X',\Q} \otimes _{D ^\dag _{\X,\Q}} G$.

I.2) Or, d'après un théorème de type $A$, 
$\Gamma (\X',\alpha ^{*} (\G))$ est $D ^\dag _{\X',\Q}$-cohérent. 
De plus, le morphisme $D ^\dag _{\X,\Q}\to D ^\dag _{\X',\Q}$ est fidèlement plat car il se déduit par extension 
du morphisme $\V \to \V'$. Cela implique alors que 
$G$ est un $D ^\dag _{\X,\Q}$-module cohérent. 

II) Soit $a$ une section globale de $\O _{\X}$ et notons $a'$ la section globale de $\O _{\X'}$ déduite.
De manière analogue à I.1), comme $\G|\X _{a}$ est $\D ^\dag _{\X _a} (\hdag Z _0 \cap X _{a})_{\Q}$-cohérent,
on vérifie 
$\Gamma (\X' _{a'},\alpha ^{*} (\G))
\riso 
D ^\dag _{\X' _{a'},\Q} \otimes _{D ^\dag _{\X _{a},\Q}} \Gamma (\X _{a},\G)$.
Comme $\alpha ^{*} (\G)$ est $\D ^\dag _{\X',\Q}$-cohérent, 
$\Gamma (\X ' _{a'}, \alpha ^{*} (\G))$ est $D ^\dag _{\X' _{a'},\Q} $-cohérent
et 
$\Gamma (\X ' _{a'}, \alpha ^{*} (\G)) \riso D ^\dag _{\X '_{a'}, \Q}  \otimes _{D ^\dag _{\X ',\Q}}  
\Gamma (\X ', \alpha ^{*} (\G))$.
Via la conclusion de I.1), il en découle le premier isomorphisme:
$\Gamma (\X ' _{a'}, \alpha ^{*} (\G)) 
\riso 
D ^\dag _{\X '_{a'}, \Q}  \otimes _{D ^\dag _{\X ,\Q}}  
G
\riso 
D ^\dag _{\X '_{a'}, \Q}  \otimes _{D ^\dag _{\X _{a},\Q}}  
(D ^\dag _{\X _{a}, \Q}  \otimes _{D ^\dag _{\X ,\Q}}  
G)
$.
Comme l'extension 
$D ^\dag _{\X _{a}, \Q}  \to D ^\dag _{\X '_{a'}, \Q} $ est fidèlement plate, 
il en résulte que le
morphisme canonique 
$D ^\dag _{\X _{a}, \Q}  \otimes _{D ^\dag _{\X ,\Q}}   G
\to 
\Gamma (\X _{a},\G)$
est un isomorphisme.
D'après le théorème de type $A$ pour les $\D ^\dag _{\X,\Q}$-modules cohérents, 
cela implique la $\D ^\dag _{\X,\Q}$-cohérence de $\G$.

III) Si $\alpha ^{*} (\G)$ est un $\D ^\dag _{\X',\Q}$-module holonome alors d'après l'étape précédente, 
$\G$ est un $\D ^\dag _{\X,\Q}$-module cohérent. 
Comme l'extension $\alpha ^{-1} \D ^\dag _{\X,\Q} \to \D ^\dag _{\X',\Q}$ 
est plate, comme $\alpha ^{*} (\G)$ est holonome, 
il en résulte que $\G$ est holonome.
\end{proof}

\begin{defi}
\label{defi-NLNL-Ddag}
Soient $T _0$ un diviseur de $X _0$ et $\G \in D ^{\mathrm{b}} _{\mathrm{coh}}(\D ^\dag _{\X} (\hdag T _0)_{\Q})$.
On dit que $\G$ vérifie la propriété $(NL\text{-}NL)$ s'il existe un diviseur $Z' $ de $X$ tel que
$U ':= X \setminus Z' $ soit affine, $Z' _0:=Z ' \otimes _{\V} k \supset T _0$ et, 
quitte à faire une extension finie de la base, 
pour tout $l \in \Z$, 
$\mathcal{H} ^{l} \G (\hdag Z _0')$ soit associé à un isocristal surconvergent sur $U' _0:=U ' \otimes _{\V} k$
qui vérifie la propriété $(NL\text{-}NL)$.
\end{defi}

\begin{exem}
Avec les notations de \ref{defi-NLNL-Ddag}, en notant $\Y:= \X \setminus T _{0}$, 
les $F$-complexes $\G$ de $F\text{-}D ^{\mathrm{b}} _{\mathrm{coh}}(\D ^\dag _{\X} (\hdag T _0)_{\Q})$ tels que 
$\G |\Y$ est holonome
vérifient la propriété $(NL\text{-}NL)$.
En effet, cela résulte de \cite[5.3.5.(i)]{Beintro2}, \cite[2.2.12]{caro_courbe-nouveau}
et du fait qu'un $F$-isocristal surconvergent sur une courbe lisse vérifie la propriété $(NL\text{-}NL)$.
Le théorème ci-dessous est donc une extension au cas sans structure de Frobenius de la conjecture {\og forte\fg} de Berthelot.
\end{exem}

\begin{theo}
\label{conjD-sansFrob}
Soient $T _0$ un diviseur de $X$ et $\G \in D ^{\mathrm{b}} _{\mathrm{coh}}(\D ^\dag _{\X} (\hdag T _0)_{\Q})$
satisfaisant à la propriété $(NL\text{-}NL)$. Alors 
$\G \in D ^{\mathrm{b}} _{\mathrm{hol}}(\D ^\dag _{\X,\Q})$.
\end{theo}

\begin{proof}
Grâce à \ref{desc-hol-chgtbase}, 
quitte à faire une extension finie de la base, 
on supposer qu'il existe un diviseur $Z' $ de $X$ tel que
$U ':= X \setminus Z' $ soit affine, $Z' _0:=Z ' \otimes _{\V} k \supset T _0$ et,
pour tout $l \in \Z$, 
$\mathcal{H} ^{l} \G (\hdag Z _0')$ soit un isocristal surconvergent sur $U ' \otimes _{\V} k$
qui vérifie la propriété $(NL\text{-}NL)$.
Il découle alors de \ref{Traduc-NLNL} que 
les espaces de cohomologie de
$f _{+} (\mathcal{H} ^{l} \G (\hdag Z _0'))$ sont de dimension finie. 
D'après \ref{j*dag-holo} et \ref{jdagE-Edag}, 
il existe un $\D ^\dag _{\X,\Q}$-module holonome $\G _{l}$ tel que
$\mathcal{H} ^{l} \G (\hdag Z _0') 
\riso 
\G _{l} (\hdag Z _0') $.
Il résulte alors de \ref{2.3.2-courbe-bis} que 
$\mathcal{H} ^{l} \G (\hdag Z _0') $ est un   $\D ^\dag _{\X,\Q}$-module holonome. 
On a ainsi établi que 
$\G (\hdag Z _0')  \in D ^{\mathrm{b}} _{\mathrm{hol}}(\D ^\dag _{\X,\Q})$.
On peut supposer $Z'_0$ et $T _0$ réduits. 
Soit $Z''_0$ tel que $Z'_0$ soit la réunion disjointe de $T _0$ et $Z''_0$. 
Comme $\G (\hdag Z _0') \riso \G (\hdag Z _0'')$ (e.g. cela découle de \cite[2.2.14]{caro_surcoherent} et du fait que 
$\G \in D ^{\mathrm{b}} _{\mathrm{coh}}(\D ^\dag _{\X} (\hdag T _0)_{\Q})$), 
on déduit du triangle de localisation de $\G$ en $Z''_0$ que 
$\R \underline{\Gamma} ^\dag _{Z''_0} (\G)\in D ^{\mathrm{b}} _{\mathrm{coh}}(\D ^\dag _{\X} (\hdag T _0)_{\Q})$.
Soient $u\,:\, \ZZ '' \hookrightarrow \X$ un relèvement de $Z'' _0\to \X$. 
Comme $Z''_0 \cap T _0 $ est vide, comme 
$u _{+} u ^{!} (\G)\riso \R \underline{\Gamma} ^\dag _{Z''_0} (\G)$ est à support dans $Z''_0$, 
il résulte alors du théorème de Berthelot-Kashiwara que $u ^{!} (\G)$ est un isocristal convergent sur $Z''_0$ (comme $Z '' _0$ est de dimension nulle, 
être un $\D^\dag _\Q$-module cohérent sur $Z '' _0$ ou un isocristal convergent sur $Z '' _0$ sont deux choses équivalentes). 
En particulier, le module $u ^{!} (\G)$ est holonome (les isocristaux convergents sont toujours holonomes car leur dual comme $\D^\dag _\Q$-module est le même que comme isocristal convergent, ce dernier n'ayant qu'un terme: voir \cite{caro_comparaison}). 
Comme $u$ est une immersion fermée,
le foncteur $u _{+}$ préserve l'holonomie (en effet, cela résulte du théorème de dualité relative et de l'exactitude de $u _{+}$).
On en déduit que 
$\R \underline{\Gamma} ^\dag _{Z''_0} (\G)\in D ^{\mathrm{b}} _{\mathrm{hol}}(\D ^\dag _{\X,\Q})$.
Via le triangle de localisation de $\G$ en $Z''_0$, il en dérive que $\G \in D ^{\mathrm{b}} _{\mathrm{hol}}(\D ^\dag _{\X,\Q})$.

\end{proof}

\begin{coro}
\label{conjD-sansFrob}
Soient $T _0$ un diviseur de $X$ et $\G$ un $\D ^\dag _{\X} (\hdag T _0)_{\Q}$-module cohérent, 
$\O _{\X} (\hdag T _0)_{\Q}$-cohérent qui 
vérifie la propriété $(NL\text{-}NL)$. Alors 
$\G $ est un $\D ^\dag _{\X,\Q}$-module holonome.
\end{coro}

\bibliographystyle{smfalpha}
\providecommand{\bysame}{\leavevmode ---\ }
\providecommand{\og}{``}
\providecommand{\fg}{''}
\providecommand{\smfandname}{et}
\providecommand{\smfedsname}{\'eds.}
\providecommand{\smfedname}{\'ed.}
\providecommand{\smfmastersthesisname}{M\'emoire}
\providecommand{\smfphdthesisname}{Th\`ese}

\bigskip
\noindent Daniel Caro\\
Laboratoire de Mathématiques Nicolas Oresme\\
Université de Caen
Campus 2\\
14032 Caen Cedex\\
France.\\
email: daniel.caro@math.unicaen.fr

\end{document}